\documentclass[10pt,letterpaper]{amsart}
\usepackage{graphicx}
\usepackage{amssymb,amscd,amsthm,amsxtra}
\usepackage{latexsym}
\usepackage{epsfig}
\usepackage{amsmath}
\usepackage{bm}
\usepackage{amssymb}
\usepackage{amsthm}
\usepackage[usenames,dvipsnames,svgnames,table]{xcolor}
\usepackage[linktocpage=true,colorlinks=true,linkcolor=Blue,citecolor=BrickRed,urlcolor=RoyalBlue]{hyperref}
\usepackage[alphabetic]{amsrefs}
\usepackage[normalem]{ulem}
\usepackage{xcolor}
\usepackage{soul}

\usepackage[colorinlistoftodos,prependcaption,textsize=tiny]{todonotes}
\usepackage{caption}

\usepackage[mathscr]{euscript} 

\usepackage{mathrsfs} 

\usepackage{pxfonts}

\usepackage[varg]{txfonts}

\numberwithin{equation}{section} 

\newtheorem{theorem}{Theorem}

\newtheorem{definition}{Definition}
\newtheorem{lemma}{Lemma}
\newtheorem{corollary}{Corollary}
\newtheorem{proposition}{Proposition}

\theoremstyle{remark}
\newtheorem{remark}{Remark}[section]

\def \suchthat {\ \big | \ }

\title[New regularity bounds]{New regularity estimates for fully nonlinear elliptic equations}
\author{ Thialita M. Nascimento $\&$ Eduardo V. Teixeira}
\address{Department of Mathematics,  University of Central Florida, Orlando, FL, USA}
\email{thnascimento@Knights.ucf.edu}
\email{eduardo.teixeira@ucf.edu}

\begin{document}
\maketitle

\date{} 

\begin{abstract} We establish new quantitative Hessian integrability estimates for viscosity supersolutions of  fully nonlinear elliptic operators. As a corollary, we show that the optimal Hessian power integrability $\varepsilon = \varepsilon(\lambda, \Lambda, n)$ in the celebrated $W^{2, \varepsilon}$--regularity estimate satisfies  
$$
	 \frac{ \left (1+ \frac{2}{3}\left(1- \frac{\lambda}{\Lambda} \right )\right )^{n-1}}{\ln n^4} \cdot \left( \frac{\lambda}{\Lambda} \right) ^{n-1} \le \varepsilon \le \frac{n\lambda}{(n-1)\Lambda +\lambda},
$$
where $n\ge 3$ is the dimension and $0< \lambda < \Lambda$ are the ellipticity constants. In particular, $\left( \frac{\Lambda}{\lambda} \right) ^{n-1} \varepsilon(\lambda, \Lambda, n)$ blows-up, as $n\to\infty$; previous results yielded fast decay of such a quantity.
The upper estimate improves the one obtained by Armstrong, Silvestre, and Smart in \cite{ASS}.


\end{abstract}

\begin{abstract}[R\'esum\'e]
Nous établissons de nouvelles estimations quantitatives d'intégrabilité pour la matrice Hessienne de sursolutions de viscosité d'opérateurs elliptiques compl\`etement non linéaires. Cons\'equemment, nous montrons que l'exposant optimal de l'intégrabilité du hessien $\varepsilon = \varepsilon(\lambda, \Lambda, n)$ dans la fameuse  estim\'ee  de régularité $W^{2, \varepsilon}$-- satisfait
$$
                \frac{ \left (1+ \frac{2}{3}\left(1- \frac{\lambda}{\Lambda} \right )\right )^{n-1}}{\ln n^4} \cdot \left( \frac{\lambda}{\Lambda} \right) ^{n-1} \le \varepsilon \le \frac{n\lambda}{(n-1)\Lambda +\lambda},
$$
où $n\ge 3$  est la dimension et $0< \lambda < \Lambda$ sont les constantes d'ellipticité. En particulier, $\left( \frac{\Lambda}{\lambda} \right) ^{n-1} \varepsilon(\lambda, \Lambda, n)$ explose, quand $n\to\infty$; les résultats précédents ont donné une décroissance rapide d'une telle quantité.
L'estimation supérieure améliore celle obtenue par Armstrong, Silvestre et Smart en \cite{ASS}.
\end{abstract}

\tableofcontents

\section{Introduction}
That solutions of the  uniformly elliptic inequality 
\begin{equation} \label{nondiv form equa}
    a_{ij}(x) \partial_{ij} u(x) \le 0 \quad \text{in} \quad B_1 \subset \mathbb{R}^n
\end{equation}
belong to the Sobolev space $W^{2, \varepsilon}$ for some $0 < \varepsilon(n, \lambda,  \Lambda) \le  1$ is a foundational result in the theory of elliptic PDEs. It was originally proven independently by Fang-Hua Lin \cite{Lin} and  Lawrence C. Evans \cite{Evans}, and later extended to the fully nonlinear framework by Luis Caffarelli \cite{Caff}, see also the book by Caffarelli and Cabre \cite{Caff-Cabre}. 

A careful look at the original constructions reveals that the exponent $\varepsilon(n, \lambda, \Lambda)$ obtained, as a function of the ellipticity ratio, $\lambda/\Lambda$, and the dimension, decays to zero very rapidly, i.e. exponentially fast. Improvements in such a decay have been recently obtained by Nam Le  in \cite{Le} and by Connor Mooney in \cite{Mooney} by means of powerful ideas {\it  ala} Cabre in \cite{Cabre} and Savin in \cite{Savin}, viz. the sliding paraboloids method to bypass the Aleksandrov-Bakelman-Pucci (ABP) estimate. This in turn improves the measure estimate of the contact points, yielding sharper control on the Hessian estimate. 

Indeed, the measure estimate obtained by Le in \cite{Le} implies that $\varepsilon(n, \lambda,  \Lambda)  \ge  c_0(n) \left(\frac{\lambda}{\Lambda} \right)^{(n+1)}$
 where  $c_0(n) \sim n^{- n/2}$. Mooney in \cite{Mooney} improves the degree of the polynomial decay, however, the dimensional constant dependence is smaller, viz.
$\varepsilon(n, \lambda,  \Lambda)  \ge  c_1(n) \left(\frac{\lambda}{\Lambda} \right)^{n-1},$ 
where $c_1(n) \sim n^{-n}$. 

 Universal Hessian estimates, such as in the $W^{2, \varepsilon}$-regularity theory, yield compactness properties of solutions which can be used to obtain further, refined regularity of solutions in more organized media, e.g. \cite{PT, SilvTeix, Teix1}  and references therein. 
 
 Quantitative estimates for fully nonlinear elliptic equations involving a large number of independent variables have become a prominent field of research as they naturally appear in models from financial engineering, machine learning and deep learning, among others, see for instance \cite{BEJ, HJE}, and references therein. In particular, improved Hessian exponent integrability for viscosity supersolutions, in very high dimensional spaces, is a relevant endeavor from both pure and applied viewpoints. 
 
From the theoretical aspect of the theory, quantitative estimates upon $\varepsilon$ relate to the Hausdorff measure of eventual singular sets of solutions to fully nonlinear equations of the form 
\begin{equation}\label{Intro F=0}
	F(D^2 u ) = 0, \quad B_1\subset{\mathbb{R}^n}.
\end{equation}
Indeed, since the program carried out by Nikolai Nadirashvili and Serge Vl\u{a}du\c{t}, \cite{NV1,NV2,NV3,NV4} it has been known that viscosity solutions to \eqref{Intro F=0} may fail to be twice differentiable and that $C^{1,\alpha}$ estimates are, in fact,  optimal.  Under $C^1$ regularity of $F$, Armstrong, Silvestre, and Smart in \cite{ASS}  proved $C^{2}$ differentiability of viscosity solutions to \eqref{Intro F=0} in the complement of a closed subset $\Sigma$. The authors actually manage to estimate the Hausdorff dimension of the singular set $\Sigma$ by $n - \varepsilon$, where $\varepsilon > 0$ is precisely the exponent of the $W^{2,\varepsilon}$-regularity theory. In \cite{ASS} the authors also show that $2/ ((\Lambda / \lambda) + 1)$ is a universal upper bound for the exponent $\varepsilon$. Upon some heuristics inferences, the authors are then led to conjecture that 
\begin{equation}\label{ASS Conj}
    \varepsilon_{\star} = \frac{2}{ (\Lambda / \lambda) + 1}
\end{equation}
is the (universal) optimal exponent in the $W^{2,\varepsilon}$ estimates. This has been an important, influential conjecture since then;  see e.g. the introduction of \cite{Le} and \cite{Teix2} for further details. 

In this paper we continue the program of obtaining improved bounds for the  optimal Hessian exponent integrability, $\varepsilon(n, \lambda,  \Lambda)$, for a viscosity supersolution of fully nonlinear elliptic equations. We  are particularly interested in problems modeled in high-dimensional spaces and our main theorem states as follows:

\begin{theorem}\label{main-thm}
    Let $u $ be viscosity solution of 
    \begin{equation}\label{The main equation}
            \mathcal{M}^{-}_{\lambda, \Lambda} (D^2 u) \le 0 \quad \text{in} \quad B_1 \subset \mathbb{R}^n,
    \end{equation}
    where $\mathcal{M}^{-}_{\lambda, \Lambda} \colon \mathcal{S}(n) \to \mathbb{R}$ denotes the Pucci extremal operator and $n\ge 3$. Then,   $u \in W^{2, \varepsilon}(B_{1/2})$, with universal estimates, for an optimal exponent $\varepsilon = \varepsilon (n, \lambda,  \Lambda)$ satisfying 
     \begin{equation}\label{main-thm-thesis}
	 \mu_n \frac{\left (1+ \left(1- \frac{\lambda}{\Lambda} \right ) \left (1 - \frac{1}{n} \right )\right )^{n-1}}{\ln n} \cdot \left( \frac{\lambda}{\Lambda} \right) ^{n-1}  \le \varepsilon \le \frac{n\lambda}{(n-1)\Lambda +\lambda},
    \end{equation}
 for a computable, increasing sequence of positive numbers, $\frac{1}{4} < \mu_n \to 1-e^{-1}$.
\end{theorem}

We actually prove a much stronger result as byproduct of the new ingredients introduced in this paper. For now, one should note that the upper bound in \eqref{main-thm-thesis} solves, in the negative, the Armstrong-Silvestre-Smart conjecture, \eqref{ASS Conj}. The lower bound, on the other hand, provides a considerable improvement on the dimensional dependence, viz. the constants $c_0(n) \sim n^{-n/2}$ and $c_1(n) \sim n^{-n}$ from Le's and Mooney's theorems respectively. In particular, if $\lambda < \Lambda$,
$$
	\lim\limits_{n\to \infty} \frac{\varepsilon(n, \lambda, \Lambda)}{\left ( \frac{\lambda}{\Lambda} \right )^{n-1}} = +\infty.
$$	
Thus, our result suggests that the sharp polynomial decay on $\varepsilon(n, \lambda, \Lambda)$, as a function of $ \frac{\lambda}{\Lambda}$, should likely be less than $(n-1)$.

We conclude this introduction by discussing the main new ideas involved in the proof of Theorem \ref{main-thm}. As usual, the lower bound is obtained from an improved $L^{\varepsilon}$ estimate of the type:
\begin{equation*}
	\left| \left\lbrace \underline{\Theta}_{u} > t \right\rbrace  \cap B_{1/2} \right|  \le C t^{-\varepsilon}
\end{equation*}
for a universal constant $C > 0$, where $\underline{\Theta}_{u}(x)$ measures the maximum aperture of concave paraboloids touching $u$ by below at $x$, see \eqref{eq2.2}.  We follow \cite{Mooney}; however the decision upon the dyadic iteration is delegated to an optimization process; a sort of {\it intrinsic scaling} for the problem. More precisely, for a new parameter $\delta > 0$, to be later selected by an optimization routine, we estimate the measure decay of the set of points where $u$ is above its lower envelope of paraboloids with opening $-(1 + \delta)^k$, for $k\in \mathbb{N}$. By choosing ``the best possible" $\delta>0$ and changing variables, we are able to show that the sharp $\varepsilon$ in the $W^{2,\varepsilon}$-regularity theory satisfies:
\begin{equation}\label{opt prob for epsilon Intro}
	\varepsilon  \ge \sup\limits_{(0,1)}\frac{\ln(1- c\gamma^n)}{\ln(1- \gamma)},
\end{equation} 
for a carefully crafted constant $0 < c(n, \lambda, \Lambda) < 1$. The bigger the $c$, the bigger the above supremum. Hence, we decide on the ``best possible" value for $c$ through yet another optimization problem which takes into account the number of possible negative eigenvalues of $D^2u$. Since viscosity supersolutions must have at least one negative eigenvalue, we show $c(n, \lambda, \Lambda) \ge \left(1 +  \left( \frac{\Lambda}{\lambda} -1\right) \frac{1}{n-1} \right)^{1-n}$, and this combined with \eqref{opt prob for epsilon Intro} yields
$$
	\varepsilon \ge  \mu_n \frac{\left (1+ \left(1- \frac{\lambda}{\Lambda} \right ) \left (1 - \frac{1}{n} \right )\right )^{n-1}}{\ln n} \cdot \left( \frac{\lambda}{\Lambda} \right) ^{n-1},
$$
which is the lower bound stated in Theorem \ref{main-thm}. 
	
The upper bound is obtained by carefully extending Armstrong-Silvestre-Smart's original example; the analysis though is considerably more involved. As a consequence of our upper estimate, if it is possible to establish an {\it adimensional} $W^{2,\varepsilon}$ regularity theory, as suggested in \cite{ASS}, then such an $\varepsilon$ must be less than or equal to $\frac{\lambda}{\Lambda}$. 

The rest of the paper is organized as follows: In Section \ref{sct Prelim} we collect preliminaries definitions and results that will assist the analysis throughout the paper. In Section \ref{sct new lemma}, we establish a new key lemma which yields faster geometric measure decay of sets. In Section \ref{sct lower estimates} we obtain improved $W^{2, \varepsilon}$ estimates and in Section \ref{sct global} we discuss how our analysis, combined with the arguments from \cite{Mooney}, yield  improved global Hessian integrability. Finally in Section \ref{stc upper estimates} we craft a viscosity supersolution whose Hessian {\it does not} belong to $L^{\frac{ n }{(n-1)\Lambda/\lambda +1}}$, fostering a the new upper bound of the theory.

\section{Preliminaries} \label{sct Prelim}
For a positive integer $n \ge 2$,  let $\mathcal{S}_n$ be the set of all real $n \times n$ symmetric matrices.  We define for $0 < \lambda \le \Lambda$ and $M \in \mathcal{S}_n$ the {Pucci's extremal operators} by
\begin{equation}\label{Pucci def}
	\mathcal{M}_{\lambda,\Lambda}^{-} (M) = \inf\limits_{\lambda \text{Id}_n \le A \le \Lambda \text{Id}_n} \textrm{trace} (AM) \hspace{0.5cm} \mbox{and} \hspace{0.5cm} \mathcal{M}_{\lambda, \Lambda}^{+} (M) =  \sup\limits_{\lambda \text{Id}_n \le A \le \Lambda \text{Id}_n} \textrm{trace} (AM)
	\end{equation}
where the $\inf$ and $\sup$ are taken over all symmetric matrices $A$ whose eigenvalues belong to $[\lambda, \Lambda].$ Given a domain $\Omega \subset \mathbb{R}^n$ and a function $u \in C(\Omega)$ we define the function $\underline{\Theta} (u,\Omega) (x_0)$ as
\begin{equation}\label{eq2.2}
	 \inf \left \{  A >0 \suchthat u(x) \ge u(x_0) + y\cdot (x-x_0) - \frac{A}{2}|x-x_0|^2 \text{ in } \Omega, \text{ for some } y \in \mathbb{R}^n \right \}.  
\end{equation}
If there is no tangent paraboloid from bellow at  $x_0$ then we say $\underline{\Theta} (u,\Omega) (x_0) = + \infty$.  For $a > 0$ and $L(x)$ an affine function, we say  
$$
	P_{L}^{a} (x) = - \frac{a}{2} |x|^2 + L(x),
$$
is  a paraboloid of opening $-a$. Following \cite{Mooney}, if $\Omega$ is a bounded, strictly convex domain and $v \in C( \bar{\Omega})$, we define the $a$-convex envelope $\Gamma_{v}^a$ on $\bar{\Omega}$ as 
$$
	\Gamma_{v}^a (x) := \sup\limits_{L} \{P_{L} ^{ a} (x) : P_{L} ^{ a } \le v \,\, \mbox{in}\,\, \bar{\Omega} \}.
$$
By convexity of $\Omega$, one verifies that $ \Gamma_{v}^a = v$ on $\partial \Omega$. Also, taking $a = 0$, one recovers the usual notion of convex envelopes by affine functions.  Next we  define 
$$
	A_a (v) : = \left \{ x \in \Omega \suchthat v(x) = \Gamma_{v}^a (x) \right \}.
$$
That is, $A_a(v) $ is the set of points in $\Omega$ where $v$ has a tangent paraboloid of opening $-a$ from below in $\Omega.$ One can check that $A_a(v)$ is closed in $\Omega$ and that $A_a(v) \subset A_b (v)$ if  $a \le b$ . Also, it is easy to verify that for any $\lambda, \, \gamma \in \mathbb{R}$ and $\beta > 0$, there holds:
\begin{equation}\label{eq2.3}
	\Gamma_{\beta v + \frac{\gamma}{2} |x|^2}^{\lambda} = \beta \Gamma_{v}^{\frac{\lambda + \gamma}{\beta}}  + \frac{\gamma}{2} |x|^2   \quad \mbox{ and } \quad  A_{\lambda} \left(\beta v + \frac{\gamma}{2} |x|^2\right) = A_{\frac{\lambda + \gamma}{\beta}} (v).
\end{equation}

Throughout the paper, $\Omega$ denotes a bounded, strictly convex domain in $\mathbb{R}^n$. The next two lemmas have been established by C. Mooney in \cite{Mooney} and they are the starting point of our analysis:

\begin{lemma}[Mooney {\cite{Mooney}*{Lemma 3.1}}] \label{l1}
	Assume  that $v \in C ( \bar{\Omega})$ and $\mathcal{M}_{ \lambda , \Lambda}^{-} ( D^2 v) \le K < \infty$. Let $P$ be a paraboloid of opening $-a < 0$ tangent from below to $\Gamma_{v}^0 $ at $x_0 \in \bar{\Omega} \setminus A_0 (v)$. Slide P up until $P + t$, $t > 0$, touches $v$ from below at $x_1 \in \bar{\Omega} $. Then, $x_1 \in \bar{\Omega} \setminus A_0 (v)$.
\end{lemma}

\begin{lemma}[Mooney {\cite{Mooney}*{Lemma 3.2}}] \label{l2}
	Assume that $v$ is convex. For any measurable set $F \subset \Omega$, let $V$ denote the set of vertices of all tangents paraboloids of opening $- a <0$   to $v$ at points in $F$. Then, $|V| \ge |F|$.
\end{lemma}

We conclude this preliminary section with a few concepts and definitions.

\begin{definition} Let $M \in \mathcal{S}_n$ be a symmetric matrix.We denote by $\mathrm{Spec}(M) := \{ \lambda_1 \le \lambda_2 \le \cdots \le \lambda_n\}$ the eigenvalues of $M$ (counting multiplicity) and 
$$
	s^{-}(M) := \# \left\{ \lambda \in \mathrm{Spec}(M) \suchthat \lambda \le 0 \right \}.
$$
\end{definition}  

\begin{definition} Let $u \in C (\bar{\Omega})$. We say $D^2 u(x)$ has at least $k$ nonpositive eigenvalues (in the viscosity sense) if for all $\varphi \in C^2(\Omega)$ such that $u-\varphi$ has a local minimum at $x_0 \in \Omega$, then
\begin{equation}\label{min neg}
	s^{-}(D^2\varphi(x_0)) \ge k.
\end{equation}
The biggest such a $k$ that makes \eqref{min neg} true for all $x_0 \in \Omega$ is called the minimal number of nonpositive eigenvalues of $D^2 u(x)$ in the viscosity sense.
\end{definition}  

We note that if  $u \in C (\bar{\Omega})$ is a viscosity supersolution of $ \mathcal{M}_{\lambda,\Lambda}^{-} ( D^2 u) \le 0$ in $\Omega$, then the minimal number of nonpositive eigenvalues of $D^2 u(x)$ is at least 1. 

We finish this preliminary section with a quick discussion upon the classical Lambert $W$ function, as it will play an important role in section \ref{sct lower estimates}. For a complex number $z$ the Lambert $W$ function is defined as the function that solves the  equation
\begin{equation}\label{Sct 2 W}
    	W(z)e^{W(z)} = z.
\end{equation}
This equation has infinitely many solutions given by $W_k (z)$, $k \in \mathbb{Z}$, called the $k^{\text{th}}$-branch of $W$. Its branches are defined through the equations 
$$
    W_{k}(z) \exp({W_k(z)}) = z, \quad W_k (z) \sim \ln_{k} (z), \quad |z| \to \infty
$$
and $\ln_{k} (z) = \ln |z| + 2 \pi i k$ is the $k^{\text{th}}$-branch of the complex natural logarithm, see for instance \cite{Corless}. When dealing with real numbers, we use the real-valued branches of the Lambert $W$ function, namely $W_0$ and $W_{-1}$.  $W_0$ is called the \textit{principal branch} and $W_{-1}$ is the only other branch that takes real values.

\begin{figure}[h]
    \centering
    \includegraphics[width=5cm, height=5cm]{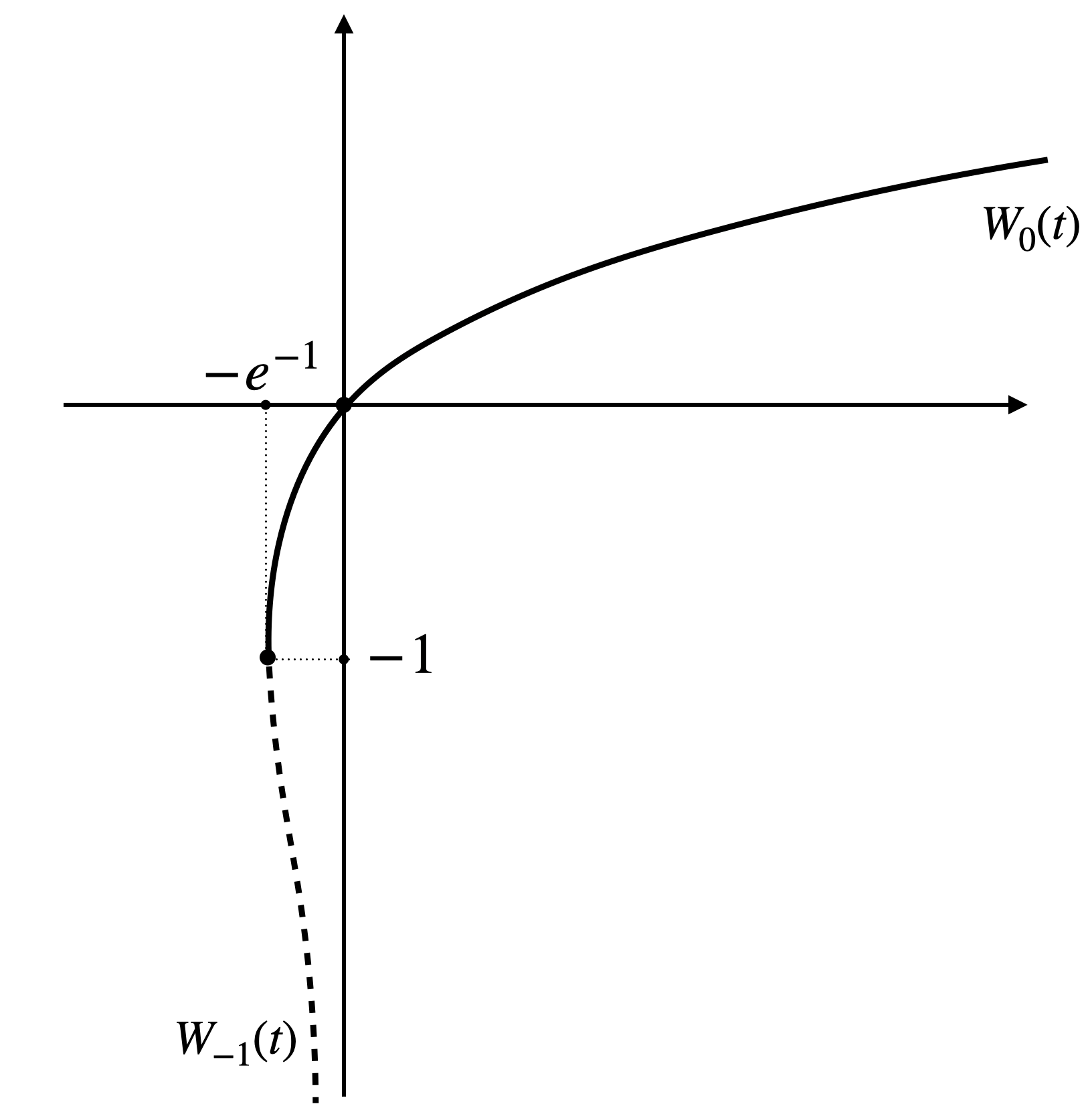}
    \caption{Real branches for the Lambert W function}
    \label{real branches for W}
\end{figure}

As the range of the function $y \mapsto ye^y$ is $[-e^{-1}, \infty)$,  for real numbers $x$ and $y$,  the equation 
$$
    	ye^y = x
$$
has a solution if, and only if $x \ge -e^{-1}$. For $x \ge 0$, the solution is $y = W_0 (x)$ and for $-e^{-1} \le x < 0$ we have two solutions given by $y = W_0 (x)$ and $y = W_{-1} (x)$. 
From the defining equation of $W$, \eqref{Sct 2 W}, it is easy to obtain a few special values such as: 
$$
    W_0 (0) = 0, \quad W_{-1} (0) = \lim \limits_{x \to 0} W_{-1} (x) = -\infty \quad \text{and} \quad W_{-1} (-e^{-1}) = W_0(-e^{-1}) = -1.
$$
From the definition we have the following identities: 
$$ 
    W_0 (xe^x) = x, \quad \text{for} \quad x \ge -1, 
$$
and 
$$ 
    W_{-1} (xe^x) = x, \quad \text{for} \quad x \le -1,
$$
Note, though, that since $f(x) = xe^x$  is not injective, then $W(f(x)) = x$ might not always hold. In fact, for a fixed $x < 0$ and $x \neq -1$, the equation $xe^x = ye^y$ has two real solutions in $y$. One of which is of course $y = x.$ And the other solution is
$$
    y = \left\{\begin{matrix}
        W_0 (x) , &x < -1\\ 
        W_{-1} (x) , & -1 < x < 0.
        \end{matrix}\right.
$$
Differentiating \eqref{Sct 2 W} and solving for $W^{\prime} (z)$ one easily gets:
$$
    \frac{dW}{dz} = \frac{1}{z + e^{W(z)}}, \quad z \neq -\frac{1}{e}.
$$    
In addition, if $z \neq 0$, there holds
$$ 
 	 \frac{dW}{dz} = \frac{W(z)}{z\left(1 + W(z) \right)}, 
$$
and since $W(0) = 0$, it follows $W_{0}^{\prime} (0) = 1.$ 

We now present lower and upper inequalities for the branch $W_{-1}$ that will be useful for us in Section \ref{sct lower estimates}. These inequalities can be derived easily from the elementary properties of the Lambert $W$ function, but we decided to include a proof as a courtesy to the readers. 

\begin{lemma}\label{ineq for W}
    For any $u \ge 0$, the following sharp inequalities hold:
    $$
        -\frac{e}{e-1} (u +1) \le W_{-1} ( - e^{-(u+1)}) \le -  ( u + 1).
    $$
    
    \begin{proof}
           First we note that, from the definition
           $$
                W_{-1} ( - e^{-(u+1)}) \exp(W_{-1} ( - e^{-(u+1)})) = - e^{-(u+1)}.
           $$
           Multiplying by $-1$ and applying the logarithm function  to both sides we obtain,
           $$
                W_{-1} ( - e^{-(u+1)}) = -u -1 - \ln (- W_{-1} ( - e^{-(u+1)})). 
           $$
           since $W_{-1} ( - e^{-(u+1)}) \le -1$, for all $u \ge 0$ then $ \ln (- W_{-1} ( - e^{-(u+1)})) \ge 0$. Therefore,
           \begin{equation}\label{upper bound for W}
                W_{-1} ( - e^{-(u+1)}) \le -u -1.
           \end{equation}
           Next, consider the function $f(u,a) =  W_{-1} ( - e^{-(u+1)}) + a( u+1)$ for $u \ge 0$ and $a \in \mathbb{R}$. Note that, $f_a( u, a) = u+ 1 > 0$ for all $u \ge 0$ which implies that $a \mapsto f(u, a)$ is an increasing function. Moreover, since $f(u, 1) =  W_{-1} ( - e^{-(u+1)}) + u +1 \le 0$ by \eqref{upper bound for W}, and $\lim\limits_{a \to \infty} f(u, a) = \infty$, then  for each fixed $u \ge0$, there exists $a(u)$ such that $f(u, a(u)) = 0$. That is, 
            $$
                a(u) = \frac{- W_{-1} ( - e^{-(u+1)})}{u+1}.
            $$
            Note next that 
            $$
                a^{\prime}(u) = \frac{ W_{-1} ( - e^{-(u+1)})}{\left(1 +  W_{-1} ( - e^{-(u+1)})\right) (u +1)^2} \left( u + 2 +  W_{-1} ( - e^{-(u+1)}) \right).
            $$
            The sign of $a^{\prime} (u)$ is dictated by the sign of the  term $ u + 2 +  W_{-1} ( - e^{-(u+1)}).$ Since   $ u + 2 +  W_{-1} ( - e^{-(u+1)}) \ge 0$ if and only if $   W_{-1} ( - e^{-(u+1)}) \ge - ( u + 2)$, and the function $xe^x$ is decreasing if $x \le -1$ then, the last inequality above holds if and only if 
            $$  
                 W_{-1} ( - e^{-(u+1)}) \exp(W_{-1} ( - e^{-(u+1)})) \le -(u+2) e^{-(u+2)}
            $$
            i.e., 
            $$  
                -e^{-(u+1)} \le -(u+2) e^{-(u+2)},
            $$
            which holds if and only if $e -2 \ge u$. Therefore, $a(u)$ is increasing in $(0, e-2)$ and decreasing in $(e-2, \infty).$ Moreover, since
            $$
                a(0) = \lim\limits_{u \to \infty} a(u) = 1
            $$
it follows that 
$$
	 1 \le a(u) \le a(e-2) = e /(e-1), 
$$
for all $u \ge 0$. Hence we can conclude
    $$
        -\frac{e}{e-1} (u +1) \le W_{-1} ( - e^{-(u+1)}) \le -  ( u + 1).
    $$
    \end{proof}
\end{lemma}
      The Lambert $W$ function, also known as the \textit{ProdctLog} function, has many other applications in pure and applied mathematics, we refer the interested readers to \cite{Corless} and \cite{Mezo} for further reading on this theme.

\section{A new Lemma} \label{sct new lemma}

In this section we establish our first new key lemma in the endeavor of improving the $W^{2, \varepsilon}$ regularity estimate. The proof follows a line of reasoning similar to \cite{Mooney}; the key main novelty though is the introduction of the new parameter $\delta$, which yields enhanced measure estimates by means of an optimization problem, to be discussed in the next section. We also carry out a more general analysis, yielding an arbitrary number of nonpositive eigenvalues.

\begin{lemma}\label{l3}
	Let $u \in C (\bar{\Omega})$ and  assume $ \mathcal{M}_{\lambda,\Lambda}^{-} ( D^2 u) \le 0$ in $\Omega$. Let $k \in [1,n]$ be the minimal number of nonpositive eigenvalues of $D^2 u(x)$. Given $a, \delta > 0$ and a measurable set $F \subset\{u > \Gamma_u ^a\}$, take the paraboloids of opening $-(1 + \delta)a $, tangents from below to $\Gamma_u ^a$ on $F$, and slide them up until they touch $u$ on a set $E $. If $E \Subset \Omega$ then, 
	\begin{equation}\label{Thesis l3}
		| A_{(1+ \delta)a} (u) \setminus A_a(u) | \ge c  \left(1 + \frac{1}{\delta}\right)^{-n} |F|,
	\end{equation}
	where $c = c(n, \lambda, \Lambda, k)$ is given by:
	\begin{equation}\label{c in l3}
		c  := \left(1 +  \left( \frac{\Lambda}{\lambda} -1\right) \frac{k}{n-k} \right)^{k-n}.
	\end{equation}
\end{lemma}
\begin{proof} We start off by defining the function 
	$$
		v : = \frac{1}{a}u + \frac{1}{2} |x|^2
	$$ 
	and noting that, in view of \eqref{eq2.3}, there holds: 
	$$
		\Gamma_{u}^a = a \Gamma_{v}^0 - \frac{a}{2} |x|^2 ,    \quad A_a(u) = A_0 (v),  \quad \mbox{and} \quad A_{(1+\delta)a}(u)  = A_{\delta}(v).
	$$
	In particular, $F \subset \{u > \Gamma_{u}^a \}$ implies $F \subset \bar{\Omega} \setminus A_0 (v).$  Next, if $P$ is a paraboloid of opening $- (1 + \delta)a < 0$, tangent to $\Gamma_{u}^a$ at $x_0 \in F$, then $P$ is also tangent  to $a \Gamma_{v}^0 - \frac{a}{2} |x|^2$  at $x_0 \in F$.  Since there exists an affine function $L(x)$ such that
	$$
		P(x) = -\frac{(1 + \delta)a}{2} |x|^2 + L(x) \le a \Gamma_{v} - \frac{a}{2} |x|^2
	$$
	with equality at $x_0$, we conclude
	$$
		-\frac{\delta}{2} |x|^2 + \frac{1}{ a} L(x) \le \Gamma_{v}^0 
	$$
	with equality at $x = x_0$. That is, there exists a paraboloid, $P_{L}^{\delta}$, of opening $-\delta < 0 $,  tangent from below to $\Gamma_{v}^0$ at $x_0 \in F$. 
	
	Next, we slide $P_{L}^{\delta}$ up until it touches $v$  from below at a point $x_1 \in \bar{\Omega}$. Lemma \ref{l1} yields $x_1 \in \bar{\Omega} \setminus A_0 (v)$ and hence, $x_1 \in A_{\delta} (v) \setminus A_0 (v).$ Collecting all such points $x_1$ in a new contact set $E \subset \bar{\Omega}$, we have,
	$$
		E \subset A_{\delta} (v) \setminus A_0 (v) =  A_{(1+ \delta)a} (u) \setminus A_a(u).
	$$	
	If $V_F$ is the set of vertices of all tangent paraboloids of opening $-\delta$ from below to $\Gamma_{v}^0$ in $F$, then by Lemma \ref{l2}, 
	\begin{equation}\label{l3-eq.1}
	|V_F| \ge |F|.
	\end{equation}
	Moreover, the vertex of the paraboloid $P_{L}^{\delta}$ remains the same after we slide it up, so  
	$$
		V_{F} = V_E,
	$$
	where $V_E$ is the set of vertices of paraboloids tangent to $v$ on the set $E$.
	
	For the time being, let us suppose that $u$ is semi-concave, and thus twice differentiable a.e in $E$. For such points $x \in E$, we have that the correspondent vertex is given by 
	$$ 
		\Phi(x) := x_V = x + \frac{1}{\delta} \nabla v(x).
	$$
	Since $\Phi (x)$ is a Lipschitz map, we can apply the area formula,
	\begin{equation}\label{l3-eq.2}
		|V_E| = |\Phi(E)| \le \int_{E} |\det (D \Phi) |  dx
	\end{equation}
	By the definition of $\Phi$ (and of $v$), the eigenvalues of $D \Phi$ are of the form 
	$$ 
		\lambda_i ^{\Phi} = (1 + \frac{1}{\delta}) + \frac{1}{\delta a} \lambda_i ^{u},\quad  i = 1, \cdots n,
	$$
	where $\lambda_i^{u}$ are the eigenvalues of $D^2 u$. Now, because $ x \in E \subset A_{(1+\delta)a}(u) \setminus A_{a}(u)$ we have 
	$$ 
		D^2 u(x) \ge -(1+\delta)a \cdot \text{Id}.
	$$ 
	Thus, the eigenvalues of $D\Phi (x)$ are all non-negative. Next, since $\mathcal{M}_{\Lambda}^{-} (D^2 u) \le 0$, for some $1\le k \le n$, the Hessian of $u$ has $k$ nonpositive eigenvalues, i.e. 
	$$
		\lambda_i^{u} \le 0  \quad \text{ for } i = 1, \cdots, k
	$$ 
	for some $ 1 \le k = k(x) \le n$. Since $\mathcal{M}_{\lambda, \Lambda}^{-}(D^2 u (x)) \le 0$, we can estimate:	
	\begin{eqnarray}
		\sum\limits_{i=1}^{n} \lambda_i ^{\Phi} &=& n \left( 1 + \frac{1}{\delta}\right)  + \frac{1}{\delta a} \sum\limits_{i=1}^{n} \lambda_i^u \nonumber \\
		&=& n \left( 1 + \frac{1}{\delta}\right)  + \frac{1}{\delta a} \left[ \sum\limits_{i=1}^{k} \lambda_i^u + \sum\limits_{i=k+1}^n \lambda_i^u \right] \nonumber \\
		& \le & n \left( 1 + \frac{1}{\delta}\right)  + \frac{1}{\delta a} \left( 1 - \frac{\Lambda}{\lambda} \right) \sum\limits_{i=1}^{k} \lambda_i^u \nonumber \\
		&=& n \left( 1 + \frac{1}{\delta}\right)  + \frac{1}{\delta a} \left( \frac{\Lambda}{\lambda} -1\right) \sum\limits_{i=1}^{k} (-\lambda_i^u). \nonumber
	\end{eqnarray}
	Moreover, as $\lambda_i^u \ge -( 1 + \delta)a$, for all $i= 1, \dots, n$ we have
	$$
		\sum\limits_{i=1}^{n} \lambda_i ^{\Phi} \le \left( 1 + \frac{1}{\delta}\right) \left[ n + \left( \frac{\Lambda}{\lambda} -1\right) k \right]. 
	$$
	Next we estimate, 
	\begin{eqnarray}\label{l3-eq.3}
			\det (D \Phi) &=& \prod\limits_{i=1}^{n} \lambda_{i}^{\Phi} \nonumber \\
			&\le & \lambda_1^{\Phi} \cdots \lambda_{k}^{\Phi} \left( \frac{(1+\delta^{-1}) \left( n + \left( \frac{\Lambda}{\lambda} -1\right) k \right)  - \sum\limits_{i=1}^{k} \lambda_{i}^{\Phi} }{n-k} 			\right)^{n-k} \nonumber \\
			&\le& \lambda_1^{\Phi} \cdots \lambda_{k}^{\Phi} \left( \frac{(1+\delta^{-1}) \left( n + \left( \frac{\Lambda}{\lambda} -1\right) k \right)  - k( \lambda_{1}^{\Phi} \cdots \lambda_{i}^{\Phi} )^{1/k} }{n-k} \right)^{n-k};
\end{eqnarray}
we have used the classical mean inequality in the last step. We note that the right hand side of \eqref{l3-eq.3} is an increasing function of $t= \lambda_1^{\Phi} \cdots \lambda_{k}^{\Phi}$ in the interval $[0, (1+\delta^{-1})^k]$ , and thus we can further estimate: 
	\begin{equation}\label{l3-eq.4}
		\det (D \Phi) \le \left(1 + \frac{1}{\delta}\right)^n \left(1 +  \left( \frac{\Lambda}{\lambda} -1\right) \frac{k}{n-k} \right)^{n-k} = c^{-1} \left(1 + \frac{1}{\delta}\right)^n,
	\end{equation}
	where $c = c(n, \lambda, \Lambda, k)$ is from \eqref{c in l3}.
Finally we can estimate:
	\begin{eqnarray}
		|F|\le |V_E| 	&\le& \int_{E} \det(D \Phi (x) ) \nonumber \\
					&\le&  c^{-1}\left(1 + \frac{1}{\delta}\right)^n  |E| \nonumber \\
					&\le & c^{-1}\left(1 + \frac{1}{\delta}\right)^n  | A_{\delta} (v) \setminus A_0 (v)| \nonumber \\
					& =& c^{-1}\left(1 + \frac{1}{\delta}\right)^n |A_{(1 + \delta)a}( u) \setminus A_a (u) |\nonumber
	\end{eqnarray}
	and the lemma is proved for semi-concave functions.
	
	For the general case, we reduce it to the previous one by using the inf-convolution 
	$$
		u_{m} : = \inf\limits_{y \in \Omega} \{u(y) + m | y -x |^2 \}.
	$$	
	Indeed, one can show, see e.g. proof of Lemma 2.1 in \cite{Savin}*{},  that $u_{m}$ are semi-concave and converge locally uniformly to $u$ in $\Omega$ as $m \to \infty$. Furthermore $u_{m}$ is a viscosity supersolution of the same equation in $\Omega^{\prime} \Subset \Omega$ and by the first part of the proof, 
	$$
		|E_{m}| \ge \beta |F|,
	$$
	where $E_{m} $ is the correspondent touching set for $u_{m}$. One can easily check that
	$$
		\limsup\limits_{m \to \infty} E_{m}  = \bigcap\limits_{m \ge 1} \bigcup\limits_{i \ge m} E_{i} \subset E,
	$$
	and thus, 
	$$
		|A_{(1 + \delta)a}( u) \setminus A_a (u) | \ge \beta |F|,
	$$
	and the proof of Lemma is complete.
\end{proof}

\begin{remark}\label{remark cstar}
As we shall discuss in the next Section, the constant  $c = c(n, \lambda, \Lambda, k)$ defined in \eqref{c in l3} plays a critical role in the endeavor of establishing quantitative lower bounds for the $W^{2,\varepsilon}$ regularity theory. The bigger the $ c(n, \lambda, \Lambda, k)$ the better the lower estimate upon Hessian integrability exponent, viz. $\varepsilon$ in the $W^{2,\varepsilon}$ regularity theory. 

Viscosity supersolutions to $ \mathcal{M}_{\lambda,\Lambda}^{-} ( D^2 u) \le 0$  must have at least one nonpositive eigenvalue, that is, one is always entitled to take $k = 1$. Geometric restrictions on the problem may impose a higher number of nonpositive eigenvalues for $D^2u$. In such situations, one can take $k$ to be greater than $1$,  but it is not obvious that this would increase the value of $c$. If not, one could simply disregard this extra information. That is, under the hypotheses of Lemma \ref{l3}, its thesis, i.e. \eqref{Thesis l3} actually holds if one substitutes  $c(n, \lambda, \Lambda, k)$ by 
\begin{equation} \label{rmk cstar eq}
	c_\star := \max \left \{ c(n, \lambda, \Lambda, 1), c(n, \lambda, \Lambda, 2),  \cdots, c(n, \lambda, \Lambda, k) \right \}.
\end{equation}
\end{remark}

Motivated by Remark \ref{remark cstar} next proposition discusses the best choice for $c_\star$. 

\begin{proposition}\label{measure estimate's corollary}
    	Let $t_0$ be the point of maximum in $[1, n)$ of the function 
    	$$
    		f(t) =  \left(1 +  \left( \frac{\Lambda}{\lambda} -1\right) \frac{t}{n-t} \right)^{n-t}.
	$$ 
	Then $t_0$ is a continuously decreasing function of the ellipticity ratio $\varrho := \Lambda/ \lambda$ and, moreover, for any $ \displaystyle \beta \in ( 1, n]$, there exists $ 1< \varrho_{\beta}$ such that 
    $$
        t_0 \left( \varrho_\beta\right) = \frac{n}{\beta}.
    $$
    
\end{proposition}
\begin{proof}
       For $\lambda, \Lambda$ and $n$ fixed, by making the change of variables
     $$
     	x= 1 +  \left( \frac{\Lambda}{\lambda} -1\right) \frac{t}{n-t},
      $$
 the maximum of $f$ in $[1, n)$ is attained at 
    $$
        t_0 = t_0 \left(\frac{\Lambda}{\lambda} \right) = \frac{n(x_0 - 1)}{(\Lambda / \lambda) -2 + x_0},
    $$
    where  $x_0$ verifies
   \begin{equation}\label{eq of x_0}
        \ln(x_0) = \frac{(\Lambda / \lambda) -2 + x_0}{x_0}. 
   \end{equation}
    That is, 
    $$
        x_0 = x_0 \left( \frac{\Lambda}{\lambda} \right) = \frac{\frac{\Lambda}{\lambda} - 2}{W_0 \left((\frac{\Lambda}{\lambda} - 2)e^{-1} \right)} \in (1, \infty )
    $$    
    where $W_0$ denoted the principal branch of the Lambert function. Since $x_0$ solves \eqref{eq of x_0} we can write
$$
    t_0 = \frac{n(x_0 - 1)}{x_0 \ln (x_0) }. 
$$
Next, we note that $t_0 (1^{+}) = n$ and $t_0$ is a continuous, decreasing function of $x_0$. Furthermore, 
\begin{eqnarray}\label{derivative of x_0}
    x_0^{\prime} \left( \frac{\Lambda}{\lambda} \right)  &=& \frac{W_0 - e^{-1}(\Lambda / \lambda  - 2) W_0^{\prime}}{ W_0^{2}} \nonumber \\
    &=&\frac{1}{1 + W_0}. 
\end{eqnarray}
By the classical properties of the Lambert function:
$$ 
    W^{\prime}(z) = \frac{W(z)}{z( 1 + W(z))}, \quad \text{for all} \quad z \neq 0, - e^{-1}.
$$
In our case, \eqref{derivative of x_0} is valid for all $\Lambda / \lambda \neq 1, 2$. However, since 
$$ 
\lim\limits_{ \Lambda/\lambda \to 2}  x_0^{\prime} \left( \frac{\Lambda}{\lambda} \right) = 1
$$
and $x_0^{\prime} $ is decreasing, we may define $ x_0^{\prime} \left( 2 \right) = 1$.  Furthermore, since $ x_0^{\prime} \left( \frac{\Lambda}{\lambda} \right) > 0$ for all $\Lambda / \lambda > 1$ then, chain rule yields

$$
    t_0^{\prime} \left(\frac{\Lambda}{\lambda} \right) = t_0^{\prime} \left( x_0 \left(\frac{\Lambda}{\lambda} \right) \right) x_0^{\prime} \left( \frac{\Lambda}{\lambda} \right) \le 0
$$
for all $\Lambda / \lambda > 1$. Therefore, $t_0$ is decreasing with respect to $\Lambda / \lambda $ in $(1, \infty )$.

Next, given $1 < \beta \le n$, we are now interested in finding $\lambda, \Lambda$ such that $t_0 ( \Lambda / \lambda ) = \beta^{-1} n$.  We then set 
$$
     \frac{n}{\beta} = \frac{n (x_0 -1)}{x_0 \ln(x_0)},
$$
which holds if and only if
\begin{equation}\label{2nd eq of x_0}
     \ln(x_0) = \beta \frac{x_0 - 1}{x_0},
\end{equation}
which is  equivalent to
$$ 
    - \beta e^{-\beta} = - \frac{\beta}{x_0} e^{- \frac{\beta}{x_0}}.
$$
The above equation has two solutions, one is of course $ x_0 (\Lambda / \lambda) = 1$  and the second is
\begin{equation}\label{2nd def for x_0}
    x_0 (\Lambda / \lambda) = \frac{-\beta} {W_0 \left( -\beta e^{-\beta} \right)}.
\end{equation}
Since $ x_0 (\Lambda / \lambda) = 1$ if only if  $\Lambda = \lambda$ and $t_0 (1^{+}) = n$, we are seeking $0 < \lambda < \Lambda$ such that $x_0 (\Lambda / \lambda)$ is given by \eqref{2nd def for x_0} and solves both $\eqref{eq of x_0}$ and \eqref{2nd eq of x_0}. 
We readily obtain
$$
	\frac{(\Lambda / \lambda) -2 + x_0}{x_0} = \beta \frac{x_0 - 1}{x_0} \implies  \Lambda / \lambda = 2 -\beta + (\beta - 1)x_0,
$$
and thus, using \eqref{2nd def for x_0}, we have

$$ 
    \Lambda_{\beta} / \lambda_{\beta} = 2 -\beta + (\beta - 1)\frac{-\beta} {W_0 \left( -\beta e^{-\beta} \right)}.
$$
    
\end{proof}
In particular, Proposition \ref{measure estimate's corollary} informs that for any $\lambda, \Lambda$ such that $ 1 \le (\Lambda / \lambda) \le  (  \Lambda_{\beta} / \lambda_{\beta} )$ we have 
$$ 
    t_0 \left( \Lambda / \lambda \right) \ge t_0( \Lambda_{\beta} / \lambda_{\beta} ) = \frac{n}{\beta}. 
$$

\section{Improved \texorpdfstring{$W^{2,\varepsilon}$}{Lg} estimates} \label{sct lower estimates}

In this section we establish improved quantitative $W^{2,\varepsilon}$ regularity estimates for viscosity supersolutions of fully nonlinear, uniformly elliptic equations. 

\begin{theorem}(Interior $W^{2, \varepsilon}$ estimate)\label{th1}
	Let $u \in C(\bar{B_1})$ such that $\mathcal{M}_{\lambda, \Lambda}^{-} (D^2 u) \le 0$ in $B_1$. Assume $D^2 u(x)$ has at least $k$ nonpositive eigenvalues, for some $k \in [1, n]$. There exists universal numbers $\gamma_0, \varepsilon \in (0,1)$ such that for all $0< \alpha < \varepsilon$,
	\begin{equation}\label{th1-eq.1}
	\left| \left\lbrace \underline{\Theta}_{u} > \left( \frac{1-\gamma_0}{\gamma_0}\right) 2^6 ||u||_{\infty} t \right\rbrace  \cap B_{1/2} \right|  \le |B_1 | t^{-\alpha}
	\end{equation}\label{th1-ep}
	 holds for $t\ge (1-\gamma_0)^{-j(\alpha) - 1}$ where $j(\alpha) = \min \{ j \in \mathbb{N} \suchthat \alpha (\varepsilon - \alpha)^{-1} \le j \} $. 
	In addition,  $\gamma_0, \varepsilon \in (0,1)$ are related by the following formula: 
	\begin{equation}\label{th1-ep}
	 	\varepsilon = \varepsilon (\gamma_0):= \sup\limits_{(0,1)}\frac{\ln(1- c\gamma^n)}{\ln(1- \gamma)} = \frac{\ln(1- c\gamma_0^n)}{\ln(1- \gamma_0)} 
	\end{equation} 
	 where $c = c_\star(n, \lambda, \Lambda, k)$ is the constant discussed in Remark \ref{remark cstar}, i.e.
	 $$
	 	\displaystyle c = \max \left \{ c(n, \lambda, \Lambda, 1),  c(n, \lambda, \Lambda, 2), \cdots c(n, \lambda, \Lambda, k) \right \},
	$$
	for $c(n, \lambda, \Lambda, i) =   \left(1 +  \left( \frac{\Lambda}{\lambda} -1\right) \frac{i}{n-i} \right)^{i-n}$.
\end{theorem}
\begin{proof} We start off by  noticing that since $\alpha < \varepsilon$ then there exists a natural number $j$ such  that 
    $$ 
    \alpha \le \frac{j}{1+j} \varepsilon
    $$
and hence the set $\{ j \in \mathbb{N} \suchthat \alpha (\varepsilon - \alpha)^{-1} \le j \}$ is nonempty. Next, for convenience, we set $\displaystyle \delta_{0} = \frac{\gamma_0}{1-\gamma_0} > 0$ and assume, with no loss of generality, that $ 0 \le u < \delta_0 2^{-4}.$  Consider the extension 
	\begin{equation*}
		\tilde{u} = \left\{
		\begin{matrix}
		\min (u, \frac{\delta_0}{4}(1 - |x|^2)),\,\,\,
		\mbox{in}\,\,\,\,\, B_1\\ 
		\frac{\delta_0}{4}(1 - |x|^2), \,\,\,\, \mbox{in} \,\, B_R \setminus B_1
	\end{matrix}\right.
	\end{equation*}
	for $R$ large. Note that $\tilde{u} \in C(\bar{B_R})$ with $\tilde{u} = u $ in $B_{\sqrt{3}/2}$ and $B_R \setminus A_{\delta_0}(\tilde{u}) \Subset  B_1$. Next we verify that
	\begin{equation}\label{th1-eq.2}
		|B_1 \setminus A_{(1 + \delta_0)^j}(\tilde{u}) | \le \left(1 - c \left(1 + \frac{1}{\delta_0}\right)^{  -n } \right)^j |B_1|, 
	\end{equation}
	for all $j\ge 0$. We argue by induction. The case  $j = 0$ follows easily. Assume \eqref{th1-eq.2} has been checked for $j$ and define 
	$$
		F_j := B_1 \setminus A_{(1 + \delta_0)^j} (\tilde{u})
	$$ 
	and consider the paraboloids of opening $-( 1 + \delta_0)^{j+1}$ tangents from below to $\Gamma_{\tilde{u}}^{(1 + \delta_0)^j}$ in $F_j$. Slide them up until they touch $\tilde{u}$ in a set 		$E_j$. From Lemma \ref{l1} and the very definition of $E_j$, we have 
	$$
		E_j \subset A_{(1 + \delta_0)^{j+1} }(\tilde{u}) \setminus A_{(1 + \delta_0)^j}(\tilde{u}) \subset \bar{B_1} \Subset B_R.
	$$ 
	 From Lemma \ref{l3} we obtain,
	 $$ 
	 	| B_1 \setminus A_{(1 + \delta_0)^j}(u)| \le c\left(1 + \frac{1}{\delta_0}\right)^{-n} \, |A_{(1 + \delta_0)^{j+1} }(u) \setminus A_{(1 + \delta_0)^j}(u) |.
	$$
	Since
	\begin{eqnarray}
		|B_1 \setminus A_{(1 + \delta_0)^{j+1}}(\tilde{u}) |&=& |B_1 \setminus A_{(1 + \delta_0)^j}(\tilde{u})| - |A_{(1 + \delta_0)^{j+1} }(\tilde{u}) \setminus A_{(1 + \delta_0)^j}(\tilde{u}) | \nonumber \\
		&\le& \left(1 - c\left(1 + \frac{1}{\delta_0}\right)^{-n}\right)|B_1 \setminus A_{(1 + \delta_0)^j}(\tilde{u})|\nonumber,
	\end{eqnarray}
	 by the induction hypothesis we have
	$$ 
		|B_1 \setminus A_{(1 + \delta_0)^{j+1}}(\tilde{u}) | \le \left(1 - c\left(1 + \frac{1}{\delta_0}\right)^{-n}\right)^{j+1} |B_1|,
	$$
	and \eqref{th1-eq.2} is proved.
	
	Continuing with the reasoning, since $\tilde{u} \le u$ and both functions agree in $B_{\sqrt{3}/2}$,  for $j \ge 0$ there holds
	$$ 
		\{\underline{\Theta}_{u} > (1 + \delta_0)^j\} \cap B_{1/2}  \subset B_1 \setminus A_{(1 + \delta_0)^j}(u).
	$$
	To conclude, we note that for any $t \ge \left(1 + \delta_0\right)^{1+ j(\alpha)}$ there exists $j \in \mathbb{N}$ such that
	\begin{eqnarray}\label{th1-eq.3}
	(1 + \delta_0)^{j + j(\alpha)} \le t < (1 + \delta_0)^{j + j(\alpha)+ 1},
	\end{eqnarray}
	and thus the following inclusions hold:
	\begin{equation}\label{th1-eq.4}
		\{\underline{\Theta}_{u} > t \} \subset \{\underline{\Theta}_{u} > (1 + \delta_0)^{j + j(\alpha)}\} \subset B_1 \setminus A_{(1 + \delta_0)^{j + j(\alpha)}} (u).
	\end{equation}
	We have proven that
	$$
		|\{\underline{\Theta}_{u} > t \} \cap B_{1/2} |\le \left(1 - c\left(1 + \frac{1}{\delta_0}\right)^{-n}\right)^{j + j(\alpha)} |B_1|,
	$$
	which, in view of \eqref{th1-eq.3}, yields the aimed estimate as long as we take
	$$ 
		\alpha \le -\frac{(j + j(\alpha))\ln(1 - c\left(1 + \frac{1}{\delta_0}\right)^{-n})}{(j + j(\alpha) + 1)\ln(1 + \delta_0)} = \frac{j + j(\alpha)}{j + j(\alpha) +1} \frac{\ln(1- c \gamma_0^n)}{\ln(1- \gamma_0)},
	$$
	for all $j\ge 1$, which is guaranteed by the definition of $ j(\alpha)$. 
	
	We conclude the proof of Theorem \ref{th1} by commenting that condition  $0\le u < \delta_0 2^{-4}$ is not restrictive. Indeed, given a generic supersolution $u$, consider 
	$$
		v = \delta_0 \frac{u + ||u||_{\infty}}{2^6 ||u||_{\infty}}
	$$ 
	and note that $ 0 \le v < \delta_0 2^{-4}$, $\, \mathcal{M}_{\lambda, \Lambda}^{-} (D^2 v) \le 0$ in $B_1$. Moreover,  
\begin{eqnarray}
    \{\underline{\Theta}_{v} > t \} &=& \{\underline{\Theta}_{u} > \delta_0^{-1} 2^6 ||u||_{\infty} t \} \nonumber \\
    & = & \left\{\underline{\Theta}_{u} > \left( \frac{1-\gamma_0}{\gamma_0}\right) 2^6 ||u||_{\infty} t \right\}
\end{eqnarray}
	
\end{proof}

The function defining $\varepsilon$ in \eqref{th1-ep}, i.e. $\varphi\colon (0,1) \to \mathbb{R}$ given as
$$
	\varphi(\gamma) := \frac{\ln(1- c\gamma^n)}{\ln(1- \gamma)},
$$ 
can be explored to provide lower bounds for $\varepsilon$. The critical point $\gamma_0$, used in the proof of Theorem \ref{th1}, satisfies 
$$
	\frac{nc\gamma_0^{n-1}}{(1 - c\gamma_0^{n})} (1- \gamma_0) = \varepsilon(\gamma_0).
$$
Such an equation comes from simplifying $\varepsilon^{\prime} (\gamma_0) = 0$ and can be numerically solved for chosen values of $n$ and $\lambda$ and $\Lambda$. 

We further note that
$$
-	\ln(1- c\gamma^n) = - c \ln\left ((1- c\gamma^n)^{1/c} \right ) \ge c \gamma^n,
$$
with asymptotic equality as $0< c\ll 1$.  Thus, we can write down  the pointwise inequality  
$$
	\varphi(\gamma)  \ge\frac{c\gamma^n}{- \ln(1-\gamma)} =: f(\gamma).
$$
The new function $f(\gamma)$ can be used as lower bounds for $\varepsilon$ and it is a slightly easier function to analyze. In particular, $\varepsilon(\gamma_0) \ge \varepsilon(\gamma_\star),$ where $\gamma_\star$ is the critical  point for  $f(\gamma)$ in $(0,1)$. 

\begin{proposition}[Improved bound in the plane]\label{th1:prop-2D} Let $u \in C(\bar{B_1})$ satisfy $\mathcal{M}_{\lambda, \Lambda}^{-} (D^2 u) \le 0$ in $B_1 \subset \mathbb{R}^2$. Then $u \in W^{2,\varepsilon}_{\text{loc}}(B_1)$, with universal estimates, for an optimal $\varepsilon>0$ satisfying:
$$
	0.407 \frac{\lambda}{\Lambda} \le \varepsilon \le \frac{2}{\frac{\Lambda}{\lambda} + 1}.
$$
\end{proposition}
\begin{proof}
	The upper bound is the one from \cite{ASS}. As for the lower bound, we initially note that $c = c(2, \lambda, \Lambda, 1)$, as defined in \eqref{c in l3}, equals $\frac{\lambda}{\Lambda}$. Hence, from the above discussion
	$$
		\varepsilon \ge \frac{\lambda}{\Lambda} \frac{\gamma^2}{- \ln(1-\gamma)},
	$$
	for all $\gamma \in (0,1)$. Evaluating at $\gamma = 0.715$ yields the result. 
\end{proof}

The analysis for higher dimensions is a bit more involved and it is the content of our next proposition:

\begin{proposition} \label{th1:cor1}
	Let $\varepsilon(\gamma_0)$  be defined as in Theorem \ref{th1}. There holds:
	$$
		\varepsilon(\gamma_0) \ge \frac{n(e-1)}{1 + ne\ln(n)} \left(\frac{n \ln(n)}{1+ n\ln(n)} \right)^n c,
	$$
	for all  $n \ge 3$. 
\end{proposition}
\begin{proof}
	By simplifying the equation $f^{\prime} (\gamma) = 0$, we see that the critical point of $f$ in $(0, 1)$, $\gamma_{\star}$, satisfies the equation  
 \begin{equation}\label{th1:cor1-eq.1}
		\frac{\gamma}{1-\gamma} = - n \ln(1-\gamma)
 \end{equation}
which  can be analytically solved in $(0, 1)$ by using the Lambert $W$ function. Indeed, \eqref{th1:cor1-eq.1} is equivalent to 
	$$
	    -\frac{1}{n(1-\gamma)} \exp{\left(-\frac{1}{n(1-\gamma)}\right)} = -\frac{1}{n} \exp{\left(-\frac{1}{n}\right)},
	$$
	whose solutions are $\gamma = 0$ and 
	$$
	    -\frac{1}{n(1-\gamma)} = W_{-1}\left(-\frac{1}{n} e^{-1/n}\right).
	$$
	Therefore, in $(0,1)$ the critical point is
	
	$$
		\gamma_{\star} : = 1 + \frac{1}{n W_{-1}(-\frac{1}{n} e^{-1/n})}.
	$$
	From \eqref{th1:cor1-eq.1} and the definition of $\gamma_{\star}$, we then have: 
	\begin{equation}\label{th1:cor1-eq.2}
		\begin{array}{lll}
			\varphi(\gamma_{\star}) &=& \displaystyle  \frac{-\ln(1-c\gamma_{\star}^n)}{-\ln(1-\gamma_{\star})}  \\
			&=& \displaystyle\left(-\ln(1-c\gamma_{\star}^n)\right) \left(\frac{n(1-\gamma_{\star})}{\gamma_{\star}}\right)  \\
			&=&\displaystyle\left( -\ln(1-c\gamma_{\star}^n) \right) \left(\frac{-1}{w_n + \frac{1}{n}}\right),
		\end{array} 
	\end{equation}
	where we have set $w_n := W_{-1}(-\frac{1}{n} e^{-1/n})$.

	Now, from Lemma \ref{ineq for W}, the  branch $W_{-1}$ of the Lambert W function satisfies the following inequalities:
	\begin{equation}\label{th1:cor1-eq.3}
	-\frac{e}{e-1} (u+1) \le  W_{-1}(-e^{-u-1}) < -(u+1) 
	\end{equation}
	for all $u > 0$.  Letting $ \displaystyle u = \ln(n) - \left(\frac{n-1}{n}\right)$, for $n\ge 2$, in \eqref{th1:cor1-eq.3} we reach
	$$ 
	    \frac{-e}{e-1} \left( \ln(n) +\frac{1}{n} \right) \le w_n < - \ln(n) - \frac{1}{n}.
	$$ 
	Add $1/n$ to both sides of the above inequalities:
	\begin{equation}\label{shap bounds applied}
	    -\frac{e}{e-1}\ln(n) - \frac{1}{n(e-1)} \le w_n \,+ \,\frac{1}{n} <  -\ln(n),
	\end{equation}
	and hence  the left hand side of \eqref{shap bounds applied} yields
	$$\frac{-1}{w_n + 1/n} \ge \frac{n(e-1)}{en\ln(n) +1}.$$
	The above inequality and  \eqref{th1:cor1-eq.2} yields
	\begin{eqnarray}
	\varphi(\gamma_{\star}) &\ge& \left(-\ln(1-c\gamma_{\star}^{n})\right) \frac{n(e-1)}{1 + ne\ln(n)}\nonumber \\
	&\ge& c\gamma_{\star}^{n}\frac{n(e-1)}{1 + ne\ln(n)}. \nonumber
	\end{eqnarray}
 Next, using the right hand side of \eqref{shap bounds applied} and the definition of $\gamma_{\star}$,  we note that  
 $$
    \gamma_{\star}^n \ge \left(\frac{n \ln(n)}{1+ n\ln(n)} \right)^n,
 $$ 
  for all $n\ge 2$ and hence 
 
 \begin{equation} \label{th1:cor1-eq.4}
     \varepsilon(\gamma_0) > \varphi(\gamma_\star) \ge   \frac{n(e-1)}{1 + ne\ln(n)} \left(\frac{n \ln(n)}{1+ n\ln(n)} \right)^n c,
     \end{equation}	
for all $n \ge 2$, and the proposition is proven.   
\end{proof}

\begin{corollary}\label{cor2} Let $\varepsilon(\gamma_0)$  be defined as in Theorem \ref{th1}. Then, 
$$
	\varepsilon(\gamma_0) \ge \frac{\tau_n}{\ln n} c,
$$
for an increasing sequence of positive numbers, $\frac{1}{4}< \tau_n \to 1-e^{-1}$, $n\ge 3$. In particular, 
$\varepsilon(\gamma_0) > \frac{c}{\ln n^4},$
for all $n \ge 3$.
\end{corollary}
\begin{proof}
It follows from \eqref{th1:cor1-eq.4} that
$$
	 \varepsilon(\gamma_0) > \varphi(\gamma_\star) \ge   \frac{n(e-1)}{1 + ne\ln(n)} \left(\frac{n \ln(n)}{1+ n\ln(n)} \right)^n \ln n \cdot \frac{c}{\ln n}.
$$
Easily one sees that
$$
	\tau_n := \frac{  n(e-1)\ln n}{1 + ne\ln(n)} \left(\frac{n \ln(n)}{1+ n\ln(n)} \right)^n
$$
is an increasing sequence converging to $1-e^{-1}$. Direct calculation yields:
$$
	\tau_3 \approx 0.2568 > \frac{1}{4},
$$
and the Corollary is proven.
\end{proof}

We now turn our attention to asymptotic lower bounds for the constant $c(n, \lambda, \Lambda, k)$ defined in \eqref{c in l3}. The case $k=\frac{n}{2}$ divides the theory. Indeed, it readily follows that:
$$
	c(n, \lambda, \Lambda, \frac{n}{2}) = \left ( \frac{\lambda}{\Lambda} \right )^{\frac{n}{2}},
$$
and thus, Theorem \ref{th1} along with Corollary \ref{cor2} implies that if $u \in C(\bar{B_1})$ such that $\mathcal{M}_{\lambda, \Lambda}^{-} (D^2 u) \le 0$ in $B_1$ and $D^2 u(x)$ has at least $\frac{n}{2}$ nonpositive eigenvalues, then $u \in W^{2,\alpha}(B_{1/2})$, with universal estimates, for all 
$$
	\alpha < \displaystyle \frac{1}{\ln n^4} \left ( \frac{\lambda}{\Lambda} \right )^{\frac{n}{2}}.  
$$
In the case  $\frac{n}{2} < k < n$, for $n$ fixed, Theorem \ref{th1} along with Corollary \ref{cor2} yields the existence of a constant $g_n>0$ such that for $u \in C(\bar{B_1})$ with $\mathcal{M}_{\lambda, \Lambda}^{-} (D^2 u) \le 0$ in $B_1$ and $D^2 u(x)$ has at least $k$ nonpositive eigenvalues, then $u \in W^{2,\alpha}(B_{1/2})$, with universal estimates, for all 
$$
	\alpha < \displaystyle g_n \cdot \left ( \frac{\lambda}{\Lambda} \right )^{n-k}.
$$
The constant $g_n$, however, goes to zero as $n \to \infty$. On the other hand, in view of Remark \ref{remark cstar}, in the case $\frac{n}{2} < k < n$, one can select
$$
	\displaystyle c = \max\limits_{i=1,\cdots k} \left \{ \left(1 +  \left( \frac{\Lambda}{\lambda} -1\right) \frac{k}{n-i} \right)^{i-n} \right \}.
$$
Proposition \ref{measure estimate's corollary} informes how to optimize such a choice.  

We now turn our attention to the important case $1\le k < \frac{n}{2}$; recall for viscosity supersolutions, with no further information, one can always take $k=1$ in the analysis.

\begin{proposition}\label{c por baixo} 
	Let $ c(n, \lambda, \Lambda, k) = \displaystyle \left(1 +  \left( \frac{\Lambda}{\lambda} -1\right) \frac{k}{n-k} \right)^{k-n}$ and assume $k < \frac{n}{2}$. Then
$$
	c = c(n, \lambda, \Lambda, k) \ge \displaystyle \left(  \frac{\Lambda}{\lambda}\right)^{k-n} \cdot   \left ( 1 + \left (\frac{n-2k}{n-k}  \right )\left (1- \frac{\lambda}{\Lambda} \right ) \right )^{n-k}.
$$
\end{proposition} 
\begin{proof}
We readily obtain
	$$
		\begin{array}{lll}
			c  &:=& \displaystyle \left(1 +  \left( \frac{\Lambda}{\lambda} -1\right) \frac{k}{n-k} \right)^{k-n}\\
			& = & \displaystyle \left(  \frac{\Lambda}{\lambda}\right)^{k-n} \cdot \displaystyle \left(  \frac{\lambda}{\Lambda} +  \left (1 - \frac{\lambda}{\Lambda} \right ) \frac{k}{n-k}\right)^{k-n}.  
		\end{array}
	$$
We can further estimate:
$$
	\begin{array}{lll}
		\displaystyle \left(  \frac{\lambda}{\Lambda} +  \left (1 - \frac{\lambda}{\Lambda} \right ) \frac{k}{n-k}\right)^{k-n} &=&
		\displaystyle \left(  1 -(1-\frac{\lambda}{\Lambda})(1- \frac{k}{n-k})  \right)^{k-n} \\
		&=& \displaystyle \left(  \frac{1}{1 -(1-\frac{\lambda}{\Lambda})(1- \frac{k}{n-k})}  \right)^{n-k} \\
		&= &\displaystyle \left(  \sum\limits_{j=0}^\infty \left [(1-\frac{\lambda}{\Lambda})(1- \frac{k}{n-k}) \right ]^j \right)^{n-k}  \\
		& > & \displaystyle \left (1+ (1-\frac{\lambda}{\Lambda})(1- \frac{k}{n-k}) \right )^{n-k} \\
		& = &  \displaystyle \left ( 1 + \left (\frac{n-2k}{n-k}  \right )\left (1- \frac{\lambda}{\Lambda} \right ) \right )^{n-k},
	\end{array}
$$
and the proof is completed. 
\end{proof}	

Finally, combining Theorem \ref{th1}, Proposition \ref{th1:cor1}, and Proposition \ref{c por baixo} we obtain:

\begin{theorem}\label{thm-main-refined} Let $u \in C(\bar{B_1})$ such that $\mathcal{M}_{\lambda, \Lambda}^{-} (D^2 u) \le 0$ in $B_1$. Let $k \in [1,n]$ be the minimal number of nonpositive eigenvalues of $D^2 u(x)$. Assume  $n\ge 3$ and $1\le k <\frac{n}{2}$. Then $u \in W^{2,\alpha}(B_{1/2})$ with universal estimates, for all 
$$
	\alpha < \displaystyle  \frac{\left ( 1 + \left (\frac{n-2k}{n-k}  \right )\left (1- \frac{\lambda}{\Lambda} \right ) \right )^{n-k}}{\ln n^4} \cdot \left(  \frac{\Lambda}{\lambda}\right)^{k-n} 
$$
In particular, if $\varepsilon(n, \lambda, \Lambda, k)$ is the quantity $\varepsilon(\gamma_0)$ defined in Theorem \ref{th1}, with $\lambda < \Lambda$, then:
$$
	\lim\limits_{n\to \infty}  \frac{\varepsilon(n, \lambda, \Lambda, k)}{\left(  \frac{\Lambda}{\lambda}\right)^{k-n}} = +\infty.
$$
\end{theorem}
In the case $k=1$, that is, with no further assumption other than  $\mathcal{M}_{\lambda, \Lambda}^{-} (D^2 u) \le 0$, the last information provided by Theorem \ref{thm-main-refined} seems to suggest that the sharp polynomial decay for Hessian integrability of $u$, as a function of $\frac{\lambda}{\Lambda}$, should indeed be smaller than $n-1$.  In \cite{ASS}, Armstrong, Silvestre, and Smart conjectured that the decay should be linear. We shall discuss this important conjecture in Section \ref{stc upper estimates}.

\section{Global \texorpdfstring{$W^{2,\varepsilon}$}{Lg} estimates} \label{sct global}

In this intermediary section we comment on global $W^{2, \varepsilon}$ estimates. We follow closely \cite[Section 5]{Mooney}; the main difference is the advent of a new intrinsic optimization procedure, introduced in the analysis, as to accelerate the geometric measure decay.

\begin{theorem}\label{th2}(Global $W^{2,\varepsilon}$ estimate)
	Assume $u \in C(\bar{B_1})$ and $\mathcal{M}_{\lambda,\Lambda}^{-} (D^2 u) \le 0 $ in $B_1$. Let $k \in [1,n]$ be the minimal number of nonpositive eigenvalues of $D^2 u(x)$. Then, for all $0 < \alpha < \varepsilon^G$
	$$\left| \left\lbrace  \underline{\Theta}_{u} >\frac{n^2 (3+\sqrt{5})^9 \|u\|_{\infty}}{2^5 (1+ \sqrt{5})^4} t \right\rbrace  \right| \le t^{-\alpha} |B_1|$$
	for $t \ge (2^{-1}(3+\sqrt{5}))^{1 + j(\alpha)}$, where 
	$j(\alpha) = \min \{ j \in \mathbb{N} \suchthat \alpha (\varepsilon - \alpha)^{-1} \le j \} $ and 
	\begin{equation} \label{epsilonG}
		\varepsilon^G =  \frac{\ln\left(1-2c\frac{(1+\sqrt{5})^n}{(3+\sqrt{5})^{n+1}}\right)}{\ln\left(2(3+\sqrt{5})^{-1} \right)},
	\end{equation}
	where $c$ is the same constant defined in \eqref{c in l3}.
	
\end{theorem}
\begin{proof}
	   First, assume that 
	\begin{equation} \label{th2-eq.1}
	    0 < u < \frac{\delta^4}{n^2 (1+\delta)^8} \le (1 + \delta)^{j} \rho_j^2
	\end{equation}
	for all $j\ge 0$, where $\delta >0$ and $0< \rho_j < (1+\delta)^{-2}$ are to be determined over the course of the proof.   
	Let 	
	$$
	    P(x) = -\frac{(1+\delta)^{j+1}}{2} |x-x_0 |^2 + y\cdot (x - x_0) + \Gamma_{u}^{(1+\delta)^j} (x_0)
	$$
	 be a paraboloid of opening $-(1+\delta)^{j+1}$ tangent to $\Gamma_{u}^{(1+\delta)^j} $ at $x_0 \in B_{1 - (1+ \delta)^2 \rho_j}$.    Since the vertex of $P$ can be written as
	$$
	    x_v = x_0 + \frac{1}{(1+\delta)^{j+1}}y,
	$$
	and $0 < \Gamma_{u}^{(1+ \delta)^j} $, we have that	
	$$ 
	    \frac{\|x_v - x_0\|^2}{2(1+ \delta)^{j+1}} \le P(x_v)  \le \Gamma_{u}^{(1+ \delta)^j} (x_v) \le u(x_v) < (1+\delta)^j \rho_j ^2
	$$
	and then,

	$$
	    \left\|x_v - x_0 \right\|^2 \le \sqrt{2} (1 + \delta)^{-1/2}\rho_j.
	$$	
	Therefore,

	$$
	    \|x_v\| \le 1 - (1+\delta)^2 \rho_j + \sqrt{2}(1+ \delta)^{-1/2} \rho_j.
	$$
	After sliding $P$ up until it touches $u$, the new contact point, $x_1$ will then satisfies 
	\begin{eqnarray}
    		\|x_1\| &\le& \|x_1 - x_v\| + \|x_v\| \nonumber \\
	       		 &\le& 1 - \left[(1+\delta)^2  - 2\cdot\sqrt{2}(1+ \delta)^{-1/2} \right]\rho_j .\nonumber
	\end{eqnarray}
	
	To ensure that $E_j  \Subset B_1$, it suffices to  have
	$$
	    0 < \left[(1+\delta)^2  - 2\cdot\sqrt{2}(1+ \delta)^{-1/2} \right] \rho_j < 1.
	$$
	
	Since $0 < \rho_j < (1 + \delta)^{-2} $, then the above inequalities hold if
	
	\begin{equation}\label{th2-eq.2}
	   	 \left[(1+\delta)^2  - 2\cdot\sqrt{2}(1+ \delta)^{-1/2} \right] > 0 \quad \iff \quad  \delta > 2^{3/5} -1.
	 \end{equation}
	For values of  $\delta$ in this range, we have, by Lemma (\ref{l3}) that,
	\begin{equation*}
	    |A_{(1+\delta)^{j+1}} \setminus A_{(1+\delta)^{k}} | \ge c  \left(1 + \frac{1}{\delta}\right)^{-n} |F_j|,
	\end{equation*}
	where $F_j : = B_{1-(1+\delta)^2\rho_j} \setminus A_{(1+\delta)^j}.$ Our next goal is then to show that
	\begin{equation}\label{th2-eq.3}
	|B_1 \setminus A_{(1+\delta)^j} | \le \left(1 - c \frac{\delta^n}{(1+\delta)^{n+1}}\right)^{j}|B_1|
	\end{equation}
	for all $j \ge 0$.  As in the interior estimate, we argue by induction. The case $k = 0$  is easy to see. For the induction step we consider two cases:

\noindent {\bf Case I:} $|F_j| \ge \frac{1}{1+\delta} |B_1 \setminus A_{(1+\delta)^j} |$. 
			In this case, it follows that
			\begin{eqnarray}\label{th2-eq.4}
	  		  	|B_1 \setminus A_{(1+\delta)^{j+1}} | &=& |B_1 \setminus A_{(1+\delta)^j} |  - |A_{(1+\delta)^{j+1} }\setminus A_{(1+\delta)^j} |\nonumber \\
	  			  &\le&\left(1 - c \frac{\delta^n}{(1+\delta)^{n+1}}\right) |B_1 \setminus A_{(1+\delta)^j} | \nonumber \\
    					&\le& \left(1 - c \frac{\delta^n}{(1+\delta)^{n+1}}\right)^{j+1} |B_1|,
			\end{eqnarray}
	using the inductive hypothesis.
	
\noindent {\bf Case II:} $|F_j| < \frac{1}{1+\delta} |B_1 \setminus A_{(1+\delta)^j} |$. In this case, 
	\begin{eqnarray}
	    |B_1 \setminus A_{(1+\delta)^{j+1}}| &\le& |B_1\setminus A_{(1+\delta)^j}| \nonumber \\
	    &\le& |F_j| + |B_1 \setminus B_{1 - (1+\delta)^2 \rho_j}| \nonumber \\
    	&<& \frac{1}{1+\delta} |B_1 \setminus A_{(1+\delta)^j} | +  |B_1 \setminus B_{1 - (1+\delta)^2 \rho_j}|, \nonumber
	\end{eqnarray}
	and therefore,
	\begin{eqnarray}
	    |B_1 \setminus A_{(1+\delta)^{j+1}}| &\le& \left(\frac{1+\delta}{\delta}\right) |B_1 \setminus B_{1 - (1+\delta)^2 \rho_j}| \nonumber \\
	    &\le& \left(\frac{1+\delta}{\delta}\right)\left[1 - \left(1 - (1+\delta)^2 \rho_j\right)^n  \right] |B_1|\nonumber \\
    	&\le& \left(\frac{1+\delta}{\delta}\right) n(1+\delta)^2 \rho_j |B_1| \nonumber
	\end{eqnarray}
	where the last inequality is obtained by using Bernoulli's inequality.  
	
	Now, we choose $\rho_j$ such that
	\begin{equation}\label{th2-eq.5}
    	\left(\frac{1+\delta}{\delta}\right) n(1+\delta)^2 \rho_j = \left(1 - c \frac{\delta^n}{(1+\delta)^{n+1}}\right)^{j+1}
	\end{equation}
	and hence,
	$$  
	|B_1 \setminus A_{(1+\delta)^{j+1}}| \le \left(1 - c \frac{\delta^n}{(1+\delta)^{n+1}}\right)^{j+1} |B_1|.
	$$
	With this choice of $\rho_j$, the induction step is satisfied in both cases and therefore, \eqref{th2-eq.3} is proved.

	Note that equation \eqref{th2-eq.5} completely determines $\rho_j$ and
	$0 < \rho_j < (1 + \delta)^{-2} $ holds for such a $\rho_j$. Moreover,
$$
    \frac{\delta^4}{n^2 (1+\delta)^8} \le (1 + \delta)^{j } \rho_j^2
$$
for all $j \ge 0$, if we take $\delta = (1+\sqrt{5})/2$. 


To conclude the proof, we now argue as in Theorem \ref{th1}. For any  $t \ge (1+\delta)^{1+ j(\alpha)}$,  there exists $j \ge 1$ such that
	\begin{eqnarray}\label{th2-eq.6}
	    (1 + \delta)^{j + j(\alpha)} \le t < (1 + \delta)^{j + j(\alpha)+ 1},
	\end{eqnarray}
	and thus the following inclusions hold:
	\begin{equation}\label{th2-eq.7}
		\{\underline{\Theta}_{u} > t \} \subset \{\underline{\Theta}_{u} > (1 + \delta)^{j + j(\alpha)}\} \subset B_1 \setminus A_{(1 + \delta)^{j + j(\alpha)}} (u)
	\end{equation}
	which in view of \ref{th2-eq.3} yields the aimed estimate if 
	
	\begin{eqnarray}
	    \alpha  &\le&  \frac{-(j + j(\alpha))\ln\left(1 - c \frac{\delta^n}{(1+\delta)^{n+1}}\right)}{(j+ j(\alpha) +1)\ln(1+\delta)} \nonumber \\
	    &= &\frac{j + j(\alpha)}{j + j(\alpha) +1} \frac{\ln\left(1-2c\frac{(1+\sqrt{5})^n}{(3+\sqrt{5})^{n+1}}\right)}{\ln\left(2(3+\sqrt{5})^{-1} \right)}.  \nonumber
	\end{eqnarray}
	And the above inequality holds for any $j \ge 1$ by the definition of $j(\alpha)$.  
	
	Finally, as in Theorem \ref{th1}, the condition \eqref{th2-eq.1} on $u$ is not restrictive. If $u$ is a generic supersolution, we consider
	\begin{eqnarray}
    	v &=& \delta^4 \frac{u + \| u \|_{\infty}}{2n^2 (1+\delta)^9 \|u\|_{\infty}} \nonumber \\
    	&=& \frac{2^5 (1+ \sqrt{5})^4}{n^2 (3+\sqrt{5})^9 \|u\|_{\infty}} (u + \|u \|_{\infty}) \nonumber
	\end{eqnarray}
	and note that $\eqref{th2-eq.1}$ holds for $v$, $\, \mathcal{M}_{\lambda, \Lambda}^{-} (D^2 v) \le 0$ in $B_1$, and   
	$$  
    	\left\{\underline{\Theta}_{v} > t\right\} = \left\{         \underline{\Theta}_{u} >\frac{n^2 (3+\sqrt{5})^9     \|u\|_{\infty}}{2^5 (1+ \sqrt{5})^4} t \right\}.
	$$
	The proof is complete.
\end{proof}

\begin{remark} Again, the bigger the constant $c$ in \eqref{epsilonG} the sharper the lower bound for $\varepsilon^G$. Hence, as in the interior case, one is entitled to choose $c$ as 
$$
	\max\limits_{i=1,\cdots k} \left \{ \left(1 +  \left( \frac{\Lambda}{\lambda} -1\right) \frac{k}{n-i} \right)^{i-n} \right \}.
$$
Proposition \ref{measure estimate's corollary} informs about such an optimization problem.
\end{remark}

\section{New upper bound} \label{stc upper estimates}
 In this section we discuss an upper bound for the Hessian integrability of viscosity supersolutions of fully nonlinear, uniformly elliptic equations. For each $n \ge3$, we shall craft a supersolution $v$ in $\mathbb{R}^n$ showing that the exponent $\varepsilon$ as stated in Theorem $\ref{th1}$ cannot be larger then $\frac{n\lambda}{(n-1)\Lambda +\lambda}$. This solves, in the negative,  Armstrong-Silvestre-Smart's conjecture, \cite{ASS}*{Conjecture 3.1}.

Let $\alpha, R>  0$ be given. We consider the function $u = u_{\alpha, R}$ defined over the $\mathbb{R}^n \setminus \{0\}$  by
	$$
		u(x) := \left \{
			\begin{array}{cll}
				R^{\alpha + 2}|x|^{-\alpha}  + \frac{\alpha}{2}|x|^2  - (1 + \frac{\alpha}{ 2})R^2  &\mbox{for} &  0 < |x| < R \\
				0 &\mbox{for} & R \le |x| < \infty,
			\end{array}
			\right.%
	$$
 Clearly $u \in C^1 (\mathbb{R}^n \setminus \{0\})$ and it is a non-negative function. Chain rule yields,  
$$ 
    \partial_i u = -\alpha R^{\alpha + 2} |x|^{-(\alpha +2)} x_i + \alpha x_i, \, \mbox{for} \, i =1, \dots , n
$$
and for all $i,j =1, \dots, n$
$$
    \partial_{ij} u = \alpha (\alpha + 2) R^{(\alpha + 2)} |x|^{-\alpha - 4} (x_i x_j ) - \left(\alpha R^{(\alpha +2)} |x|^{-\alpha - 2}   -  \alpha  \right) \delta_{ij},
$$    
where $\delta_{ij}$ is the Kronecker delta. Therefore, 
	$$
		D^2 u(x) = \alpha (\alpha + 2) R^{\alpha + 2} |x|^{-\alpha - 4} x \otimes x^T  - \alpha |x|^{-\alpha - 2}\left(R^{\alpha + 2} - |x|^{\alpha + 2}\right) I_n
	$$
	for all $0 < |x| < R$, where $x \otimes x^T$ is the vector direct product of $x$ and itself. 
Therefore the eigenvalues of $D^2 u(x)$ are 
	\begin{eqnarray}\label{eigenvalues}
	    \lambda_1 = \cdots = \lambda_{n-1} &=& - \alpha |x|^{-\alpha - 2}\left(R^{\alpha + 2} - |x|^{\alpha + 2}\right) \quad \text{ and } \nonumber \\ 
	    \lambda_n &=& \left(\alpha(\alpha + 1)R^{\alpha + 2} |x|^{-\alpha - 2} + \alpha \right) .
     \end{eqnarray}	
	in $0 < |x| < R$. 
	Thus, we can estimate, 
	\begin{eqnarray}\label{Pucci - D^2 u's ineq}
	    \mathcal{M}_{\lambda, \Lambda}^{-} \left(D^2 u\right )&=&  \Lambda \cdot \left(\sum_{i=1}^{n-1} \lambda_i \right) + \lambda \cdot \lambda_n \nonumber\\
		&\le &-\Lambda(n-1)\alpha |x|^{-\alpha - 2}\left(R^{\alpha + 2} - |x|^{\alpha + 2}\right) +  \lambda \left(\alpha(\alpha + 1)R^{\alpha + 2} |x|^{-\alpha - 2} + \alpha \right)  \nonumber \\
		&=&-\alpha|x|^{-\alpha - 2} R^{\alpha +  2} \left(\Lambda (n-1) - \lambda (\alpha + 1) \right)  
	+ \alpha \Lambda (n-1) + \lambda \alpha \nonumber\\
		&\le& \Lambda n \alpha
	\end{eqnarray}
	provided that 
	$$
	    \Lambda (n-1) - \lambda (\alpha + 1) \ge 0.
	$$
	That is, 
	\begin{equation}\label{condition in alpha}
     0 < \alpha \le (n-1) \frac{\Lambda}{\lambda}  - 1
\end{equation}	
	Since  $u \equiv 0$ in $\mathbb{R}^n \setminus B(0, R) $, then \eqref{Pucci - D^2 u's ineq} actually holds (in the viscosity sense) in the entire set $\mathbb{R}^n \setminus \{ 0\}$.

	Our example is now constructed, as in \cite{Mooney},  as follows: For $R > 0$ small enough,  define the function  $ v = v_{\alpha, R}$ by 
	$$
		v(x) =  -|x|^2   + \sum\limits_{y} \min \left(1, \frac{\lambda}{\Lambda \alpha} u(x - 2Ry)\right)
	$$
	where $B(2Ry, R)$ are disjoint balls and $y \in \mathbb{Z}^n$.
	
	 One can easily check that 
	\begin{equation}
		\sup_{B_1} |v| \le 1 \quad \text{and} \quad   \mathcal{M}_{\lambda, \Lambda}^{-} \left( D^2 v \right) \le 0 \quad \text{in} \quad \mathbb{R}^n .
	\end{equation}

	Let $N$ denote a   generic open neighborhood of  $x \in  B(0, R)\setminus \{ 0 \}   $.  In $N$, we can estimate: 
	$$
		\underline{\Theta} (u, N) (x) \ge -\lambda_1 = \alpha |x|^{-\alpha - 2}\left(R^{\alpha + 2} - |x|^{\alpha + 2}\right)
	$$
	
	Therefore, we conclude that if $x \in  B(0, R/2)\setminus \{ 0 \} $
	\begin{equation}\label{eq_example.4}
		\underline{\Theta} (u, N) (x) \ge c \alpha R^{\alpha + 2}|x|^{-\alpha - 2} 
	\end{equation}
	for $c$ given by 
	$$
		\displaystyle c = \left(1 - \frac{1}{2^{\alpha + 2}} \right).
	$$

	 By choosing $R> 0$ small enough, we can assure, 
	$$
		\min \left(1, \frac{\lambda}{\Lambda \alpha} u(x - 2Ry)\right) = \frac{\lambda}{\Lambda \alpha} u(x - 2Ry) 
	$$
	whenever $ |x - 2Ry| \ge \tilde{c} R^{(\alpha +  2  )/\alpha}$, where $\displaystyle  \tilde{c} = \left(\lambda/ \Lambda \alpha \right)^{1/\alpha}.$ 
	
	Moreover, for any neighborhood $N$ of $x$ in $\left\{ \tilde{c} R^{(\alpha + 2)/\alpha} \le |x - 2Ry| \le R/2 \right\} $ we have, as in (\ref{eq_example.4}),
	\begin{eqnarray}\label{eq_example.5}
		\underline{\Theta} (v,  N) (x) &\ge& 2 - \frac{\lambda}{\Lambda \alpha} \lambda_1^{u} \nonumber\\
		&\ge&  c \alpha R^{\alpha + 2}|x-2Ry|^{-\alpha - 2}
	\end{eqnarray}
	where $\lambda_1^u$ is a negative eigenvalue of $D^2u$.

	Next, if we take  $\varepsilon > 0$ such that    
\begin{equation}\label{condition on varepsilon}
    \left( (n-1)\frac{\Lambda}{\lambda} + 1\right)  \varepsilon > n,
\end{equation}
we can choose  $\alpha > 0 $ such that 
 \begin{equation}\label{2nd condtion on alpha}
     	(n-1)\frac{\Lambda}{\lambda} -1 \ge \alpha > \frac{n}{\varepsilon} - 2.
 \end{equation}

   Furthermore,  for any ball $ B(2Ry, R) \subset  B_{1/2}$, owing to (\ref{eq_example.5}) we can estimate
	\begin{eqnarray}\label{first integral estimate}
		\int\limits_{B(2Ry, R)} |\underline{\Theta} (v,  B_1) (x) |^{\varepsilon} \, dx 
		&\ge& c R^{(\alpha + )\varepsilon}\int\limits_{  \left(B_{ R/2}(2Ry) \setminus B_{\bar{c}R^{(\alpha + 2)/\alpha} }(2Ry)\right)} (|x-2Ry|)^{-(\alpha + 2) \varepsilon}  \,dx \nonumber \\
		&=& c R^{(\alpha + 2)\varepsilon}\int\limits_{ \left(B_{ R/2} \setminus B_{\bar{c}R^{(\alpha + p)/\alpha} }\right)} (|x|^{-(\alpha + 2)})^{\varepsilon} \,dx \nonumber \\
		& = & |\partial B_1 | c R^{(\alpha + 2)\varepsilon} \int\limits_{\bar{c}R^{(\alpha + p)/\alpha}}^{R/2} t^{-(\alpha + 2)\varepsilon + n-1} \, dt, \nonumber \\
		&=& \frac{C}{(\alpha + 2)\varepsilon - 2} \left[ c R^{(\alpha + 2)(n - 2\varepsilon) /\alpha}  - \frac{R^n}{2^{n- (\alpha + 2)\varepsilon}}\right],
	\end{eqnarray}
	

		
	
	The idea now is to estimate $ \displaystyle \| \underline{\Theta}(v,  B_1 ) \|_{L^{\varepsilon} ( B_{1/2})} $ by counting the number of 
	disjoint balls $B(2Ry, R)$  that fits inside $B_{1/2}$.     Let $R > 0$ be such that 
	$$
		\displaystyle \frac{1}{2\sqrt{n}} < \frac{1}{8\sqrt{n} R} = m + \frac{1}{2\sqrt{n}},
	$$ 
	for some $m \in \mathbb{N}$ large. Then, if we choose $ y = (y_1, \dots , y_n) \in \mathbb{Z}^n$ such that 
	$$
		0 < |y_j| \le \frac{1 - 4 R} {8\sqrt{n} R} = m, \quad j =1, \dots, n
	$$
	then $B(2Ry, R)$ is entirely inside $B(0, 1/2)$, and with this choice of $R$ and $y \in \mathbb{Z}^n$, there exist at least $ \displaystyle \left( \frac{1 - 4 R} {8\sqrt{n} R}\right) ^n $ disjoint balls of the form $B(2Ry, R)$, with $y \in \mathbb{Z}^n$ inside $B(0, 1/2)$. 
	

	Finally we estimate:
	\begin{eqnarray}
		\int_{B_{1/2}} | \underline{\Theta} (v, B_1) |^{\varepsilon} \,dx 
		&\ge& C\left( \frac{1 - 4 R} {8\sqrt{n} R}\right) ^n \left[ c R^{(\alpha + 2)(n - 2\varepsilon) /\alpha}  - \frac{R^n}{2^{n- (\alpha + 2)\varepsilon}}\right] \nonumber \\
		&=& C (1- 4 R)^n  \left[ c R^{2(n - (\alpha + 2)\varepsilon)  /\alpha}  - 2^{-n + (\alpha + 2)\varepsilon}\right], \nonumber
	\end{eqnarray}
	and since, again in view of \eqref{2nd condtion on alpha}, the exponent $2 (n - (\alpha + 2) \varepsilon )/ \alpha $ is negative, we conclude 
	$$
	 	|| \underline{\Theta}(v,  B_1)||_{L^{\varepsilon}( B_{1/2})} \to \infty \quad \text{ as }  \quad R \to 0.
	$$ 	
	In conclusion, as the decay proven in \eqref{th1-eq.1} implies $ \underline{\Theta}(u, B_1) \in L^{\hat{\varepsilon}} (B_{1/2}) $,  for any $ 0 < \hat{\varepsilon} < \varepsilon $, the example we just crafted show that, in dimension $n\ge 3$, the $W^{2, \varepsilon}$ estimates as stated in Theorem \ref{th1} cannot hold (universally) for any $\varepsilon>0$ large enough so that
	$$
		\left((n-1)\Lambda /\lambda + 1\right) \varepsilon > n.
	$$
	That is,  the $W^{2,\varepsilon}$ regularity theory for viscosity supersolutions requires  
	$$	
		\varepsilon \le \frac{n}{(n-1)\Lambda/ \lambda + 1},
	$$
	for any $n \ge 3$. This quantity is strictly less than the one Conjectured in \cite{ASS}.

\bibliographystyle{amsplain, amsalpha}

\end{document}